\documentclass[12pt]{amsart}
\usepackage{amsmath, amsthm, amssymb,amscd, mathrsfs, mathtools,tikz-cd}

\usepackage[all]{xy}

\usepackage[bookmarks=true,
            bookmarksnumbered=true, breaklinks=true,
            pdfstartview=FitH, hyperfigures=false,
            plainpages=false, naturalnames=true,
            colorlinks=true,pagebackref=true,
            pdfpagelabels]{hyperref}

\newtheorem{thm}{Theorem}[section]
\newtheorem{pro}[thm]{Proposition}
\newtheorem{lem}[thm]{Lemma}
\newtheorem{cor}[thm]{Corollary}
\newtheorem{conj}[thm]{Conjecture}

\theoremstyle{definition}

\newtheorem{rmk}[thm]{Remark}

\newtheorem{ex}[thm]{Example}

\allowdisplaybreaks

\newcommand\Hom{\operatorname{Hom}}

\def\Aut{\operatorname{Aut}}
\def\Inn{\operatorname{Inn}}

\def\Conj{\operatorname{Conj}}
\def\Core{\operatorname{Core}}

\title{Idempotents, free products and quandle coverings}

\author{Mohamed Elhamdadi} 
\address{Department of Mathematics, 
University of South Florida, Tampa, FL 33620, U.S.A.} 
\email{emohamed@math.usf.edu} 

\author{Brandon Nunez}
\address{Department of Mathematics, 
University of South Florida, Tampa, FL 33620, U.S.A.} 
\email{nunez7@usf.edu} 

\author{Mahender Singh}
\address{Department of Mathematical Sciences,
Indian Institute of Science Education and Research (IISER) Mohali, Sector 81, SAS Nagar, P O Manauli, Punjab 140306, India} 
\email{mahender@iisermohali.ac.in} 

\author{Dipali Swain}
\address{Department of Mathematics, 
University of South Florida, Tampa, FL 33620, U.S.A.} 
\email{dipaliswain@usf.edu} 

\begin{document}

\maketitle

\begin{abstract}
In this paper, we investigate idempotents in quandle rings and relate them with quandle coverings. We prove that integral quandle rings of quandles of finite type that are non-trivial coverings over nice base quandles admit infinitely many non-trivial idempotents, and give their complete description. We show that the set of all these idempotents forms a quandle in itself. As an application, we deduce that the quandle ring of the knot quandle of a non-trivial long knot admit non-trivial idempotents. We consider free products of quandles and prove that integral quandle rings of free quandles have only trivial idempotents, giving an infinite family of quandles with this property. We also give a description of idempotents in quandle rings of unions and certain twisted unions of quandles. 
\end{abstract}

\section{Introduction}
Axiomatisation of the three Reidemeister moves of planar diagrams of knots and links in the 3-space led to algebraic structures known as quandles \cite{MR2628474, MR0672410}. Besides being fundamental to knot theory, these structures arise in a variety of contexts such as set-theoretic solutions to the Yang-Baxter equation \cite{MR2128041}, Yetter-Drinfeld Modules \cite{MR2153117}, Riemannian symmetric spaces \cite{MR0217742}, Hopf algebras \cite{MR1994219} and mapping class groups \cite{MR2699808, MR1967241, MR2377276}, to name a few. Although link quandles are strong invariants of links, it is challenging to determine whether two quandles are isomorphic. This has motivated search for newer invariants of quandles themselves.
\par
In an attempt to bring ring theoretic techniques to the study of quandles, a theory of quandle rings analogous to the classical theory of group rings has been proposed in \cite{MR3977818}, where several interconnections between quandles and their associated quandle rings have been explored. Functoriality of the construction implies that morphisms of quandle rings give a natural enhancement of the well-known quandle coloring invariant of knots and links. Quandle rings of non-trivial quandles are non-associative, and it has been proved in \cite{MR3915329} that these rings are not even power-associative, which is the other end of the spectrum of associativity. Furthermore, quandle rings of non-trivial quandles over rings of characteristic more than three cannot be alternative or Jordan algebras  \cite{BPS1}. Examples of non-isomorphic finite quandles with isomorphic quandle rings have been given in \cite{MR3915329}. It has been proved that if two finite quandles admit doubly transitive actions by their inner automorphism groups and have isomorphic quandle rings, then the quandles have the same number of orbits of each order \cite{MR3915329}. A homological study of quandle rings has been initiated in a recent work \cite{EMSZ2022}, where a complete characterization of derivations of quandle algebras of dihedral quandles over fields of characteristic zero has been given, and dimension of the Lie algebra of derivations has been investigated. Zero-divisors in quandle rings have been investigated in \cite{BPS1} using the idea of orderability of quandles. It has been proved that quandle rings of left or right orderable quandles which are semi-latin have no zero-divisors. In particular, integral quandle rings of free quandles have no zero-divisors. Recent results suggest an analogue of Kaplansky's zero-divisor conjecture for quandles, which states that the quandle ring of a non-inert semi-latin quandle over an integral domain has no zero-divisors \cite{BPS1}. Units in group rings play a fundamental role in the structure theory of group rings. In contrast, it turns out that idempotents are the natural objects in quandle rings since each quandle element is, by definition, an idempotent of the quandle ring. In a recent work \cite{ENS2022}, idempotents of quandle rings have been used for constructing proper enhancements of the well-known quandle coloring invariant of knots and links in the 3-space. In this paper, we carry out a detailed investigation of idempotents in quandle rings with a focus on their precise computations. We note that idempotents of integral quandle rings of all three element quandles and the dihedral quandle of order four have been computed in \cite{BPS1}. We exploit the idea of quandle coverings as developed by  Eisermann in \cite{MR1954330, MR3205568}, where a Galois correspondence between connected coverings and subgroups of the fundamental group of the quandle  has been established. We determine the complete set of idempotents in the quandle ring of a quandle of finite type which is a non-trivial covering over a nice base quandle. We also investigate idempotents in quandle rings of free products including free quandles.
\par

Throughout the paper, we assume that our coefficient ring $\mathbf{k}$ is an integral domain with unity and has more than two elements.
\par

The paper is organised as follows. Section \ref{review} reviews basic ideas about quandle rings. In Section \ref{idempotents in quandle rings}, we begin by giving a basic sufficient condition on a quandle $X$ for its quandle ring $\mathbf{k}[X]$ to admit non-trivial idempotents (Proposition \ref{trivial subquandle idempotent}). We prove that if $G$ is an abelian group without 2 and 3-torsion, then $\mathbb{Z}[\Core(G)]$ has no non-trivial idempotent built up with at most three distinct basis elements (Proposition \ref{no three elements prop}). This along with a computer assisted check for quandles of order less than seven suggests that integral quandle rings of finite latin quandles have only trivial idempotents (Conjecture \ref{Main conjecture}). In Section \ref{coverings and idempotents}, we relate idempotents with quandle coverings.  We prove that if $L$ is a non-trivial long knot, then the quandle ring $\mathbf{k}[Q(L)]$ of its knot quandle $Q(L)$ has non-trivial idempotents (Proposition \ref{idempotents long knots}).  As one of the main results, we prove that if $p:X \to Y$ is a non-trivial quandle covering such that $X$ is of finite type and $\mathbf{k}[Y]$ has only trivial idempotents, then $\mathbf{k}[X]$ has many non-trivial idempotents, and give their precise description (Theorem \ref{quandle covering theorem}). As an application, we deduce that the set of all idempotents of the quandle ring of such a quandle form a quandle in itself with respect to the ring multiplication (Corollary \ref{idempotens form quandle}). We also prove that the quandle ring of a non-trivial quandle covering has right zero-divisors (Proposition \ref{zero-divisors for coverings}). In Section \ref{sec idempotents in free products}, we consider free products of quandles. We overcome the lack of associativity in quandles through an appropriate length function for elements in free products. As the second main result, we prove that integral quandle rings of free quandles have only trivial idempotents (Theorem \ref{idempotents in free products}). This gives, for the first time, an infinite family of quandles whose integral quandle rings have only trivial idempotents. As a consequence we deduce that the automorphism group of the integral quandle ring of a free quandle is isomorphic to the welded braid group on the same number of strands as the rank of the quandle (Corollary \ref{auto quandle ring free products}). In Section \ref{idempotents in unions}, we compute idempotents in unions (Proposition \ref{idempotent union}) and certain twisted unions of quandles (Proposition \ref{idempotens in trivial unions}). We conclude the paper with some remarks and open problems arising from this work.
\medskip

\section{Review of quandle rings}\label{review}
A {\it quandle} is a non-empty set $X$ with a binary operation $*$ such that each right multiplication $S_y: X \rightarrow X$ given by $S_y(x) =x*y$ for $x \in X$, is an automorphism of $X$ fixing the element $y$. A quandle is completely determined by the set of right multiplications by its elements. Note that right multiplications being homomorphisms of $X$ is equivalent to the right distributivity axiom $$(x*y)*z=(x*z)*(y*z)$$ for all $x, y, z \in X$. It turns out that quandle axioms are simply algebraic formulations of the three Reidemeister moves of planar diagrams of knots and links in the 3-space \cite{MR2628474, MR0672410}.  A quandle $X$ is called {\it trivial} if $x*y=x$ for all $x, y\in X$. The automorphism group of a quandle $X$ contains a normal subgroup  $\Inn(X)= \langle S_x \mid x \in X \rangle$ generated by all right multiplications, called the {\it inner automorphism group} of $X$. A quandle is {\it faithful} if the natural map $X \to \Inn(X)$ given by $x \mapsto S_x$ is injective.  An orbit of the natural $\Inn(X)$-action on $X$ is called a {\it connected component} of $X$. 
\par

Conjugacy classes in groups are a rich source of quandles. Each group $G$ can be turned into the {\it conjugation quandle} $\Conj(G)$ with the binary operation $x*y= yx y^{-1}$ of conjugation. Similarly, every group $G$ can be turned into a quandle $\Core(G)$ by setting $x*y=yx^{-1}y$, called the {\it core quandle} of $G$. For abelian groups (written additively), the quandle operation becomes $x*y=2y-x$. In particular, the cyclic group of order $n\ge 2$ gives the dihedral quandle $R_n=\{0, 1, 2, \ldots, n-1 \}$  of order $n$. 
\par
A quandle $X$ is called {\it latin} if each left multiplication $L_x: X \to X$ given by $L_x(y)=x*y$ for $y \in X$, is a bijection, and called {\it semi-latin} if each $L_x$ is an injection. For example, dihedral quandles of odd orders and commutative quandles are latin. Obviously, every latin quandle is semi-latin, but not conversely. For example $\Core(\mathbb{Z})$ is semi-latin but not latin. A quandle $X$ is called {\it medial} if $$(x*y)*(z*w)=(x*z)*(y*w)$$ for all $x, y, z, w \in X$. These are precisely the quandles for which the map $X \times X \to X$ given by $(x, y) \mapsto x*y$ is a quandle homomorphism, where $X \times X$ is equipped with the product quandle structure. We say that a quandle $X$ is of {\it finite type} if for each $x \in X$, the inner automorphism $S_x$ has finite order, say, $n_x$. Note that homomorphic image of a quandle of finite type is again a quandle of finite type. For example, all finite and involutory quandles are of finite type.
\par

It follows from the right distributivity axiom in a quandle $X$ that 
\begin{equation}\label{left association identity}
x*^\epsilon\left(y*^{\mu} z\right)=\left(\left(x*^{-\mu} z\right)*^{\epsilon} y\right)*^ {\mu} z
\end{equation}
for all $x,y,z \in X$ and $\epsilon, \mu \in\{-1,1\}$. For ease of notation, we write a left-associated product
\begin{equation*}
\left(\left(\cdots\left(\left(x_0*^{\epsilon_1}x_1\right)*^{\epsilon_2}x_2\right)*^{\epsilon_3}\cdots\right)*^{\epsilon_{n-1}}x_{n-1}\right)*^{\epsilon_n}x_n
\end{equation*}
simply as 
\begin{equation*}
x_0*^{\epsilon_1}x_1*^{\epsilon_2}\cdots*^{\epsilon_n}x_n.
\end{equation*}
A repeated use of \eqref{left association identity} gives the following result \cite[Lemma 4.4.8]{Winker1984}.

\begin{lem}\label{Lem:the canonical left associated form}
Let $X$ be a quandle. Then the product 
\begin{equation*}
(x_0*^{\epsilon_1}x_1*^{\epsilon_2}\cdots*^{\epsilon_m}x_m)*^{\mu_0}\left(y_0*^{\mu_1}y_1*^{\mu_2}\cdots*^{\mu_n}y_n\right)
\end{equation*}
of two left-associated expressions in $X$ is the left-associated expression
\begin{equation*}
x_0*^{\epsilon_1}x_1*^{\epsilon_2}\cdots*^{\epsilon_m}x_m*^{-\mu_n}y_n*^{-\mu_{n-1}}y_{n-1}*^{-\mu_{n-2}}\cdots*^{-\mu_1}y_1*^{\mu_0}y_0*^{\mu_1}y_1*^{\mu_2}\cdots*^{\mu_n}y_n.
\end{equation*}
\end{lem}

The 	quandle axioms imply that each element of a quandle $X$ has a canonical left-associated expression $x_0*^{\epsilon_1}x_1*^{\epsilon_2}\cdots*^{\epsilon_n}x_n$, where $x_0\neq x_1$, and if $x_i=x_{i+1}$ for any $1 \le i \le n-1$, then $\epsilon_i= \epsilon_{i+1}$.
\medskip

Let $(X, *)$ be a quandle and $\mathbf{k}$ an integral domain with unity {\bf 1}.  Let $e_x$ be a unique symbol corresponding to each $x \in X$. Let $\mathbf{k}[X]$ be the set of all formal expressions of the form $\sum_{x \in X }  \alpha_x e_x$, where $\alpha_x \in \mathbf{k}$ such that all but finitely many $\alpha_x=0$.  The set $\mathbf{k}[X]$ has a free $\mathbf{k}$-module structure with basis $\{e_x \mid x \in X \}$ and admits a product given by 
 $$ \Big( \sum_{x \in X }  \alpha_x e_x \Big) \Big( \sum_{ y \in X }  \beta_y e_y \Big)
 =   \sum_{x, y \in X } \alpha_x \beta_y e_{x * y},$$
where $x, y \in X$ and $\alpha_x, \beta_y \in \mathbf{k}$. This turns $\mathbf{k}[X]$ into a ring (rather a $\mathbf{k}$-algebra) called the {\it quandle ring} of $X$ with coefficients in $\mathbf{k}$. Even if the coefficient ring $\mathbf{k}$ is associative, the quandle ring $\mathbf{k}[X]$ is non-associative when $X$ is a non-trivial quandle. The quandle $X$ can be identified as a subset of $\mathbf{k}[X]$ via the natural map $x \mapsto {\bf 1}e_x=e_x$. 
\par
The surjective ring homomorphism $\varepsilon: \mathbf{k}[X] \rightarrow \mathbf{k}$ given by $$\varepsilon \Big(\sum_{x \in X }  \alpha_x e_x \Big)=\sum_{x \in X }  \alpha_x$$
is called the {\it augmentation map}. The kernel of $\varepsilon $ is a two-sided ideal of $\mathbf{k}[X]$, called the {\it augmentation ideal} of $\mathbf{k}[X]$. Throughout the article, we make a distinction between the product in a quandle and the product in its associated quandle ring. 
\medskip

\section{Idempotents in quandle rings}\label{idempotents in quandle rings}
Let $X$ be a quandle and $\mathbf{k}$ an integral domain with unity.  A non-zero element $u \in \mathbf{k}[X]$ is called an {\it idempotent} if $u^2=u$. We denote the set of all idempotents of $\mathbf{k}[X]$ by $\mathcal{I}\big(\mathbf{k}[X]\big)$, that is,
$$
\mathcal{I}\big(\mathbf{k}[X] \big)=\big\{ u  \in \mathbf{k}[X]\; | \; u^2=u\big \}.
$$
It is clear that the basis elements $\{e_x \mid x \in X\}$ are idempotents of $\mathbf{k}[X]$, and we refer them as {\it trivial idempotents}.  A non-trivial idempotent  is an element of $\mathbf{k}[X]$ that is not of the form $e_x$ for any $x \in X$. Clearly, if $Y$ is a subquandle of $X$, then $\mathcal{I} \big(\mathbf{k}[Y] \big) \subseteq \mathcal{I} \big(\mathbf{k}[X] \big)$. The construction of a quandle ring is functorial. Thus, a quandle homomorphism $\phi: X \rightarrow X'$ induces a ring homomorphism $\hat{\phi}: \mathbf{k}[X] \rightarrow \mathbf{k}[X']$, and hence  $\hat{\phi}$ maps $ \mathcal{I} \big(\mathbf{k}[X] \big)$ into $\mathcal{I} \big(\mathbf{k}[X'] \big)$. Since the augmentation map $\varepsilon: \mathbf{k}[X] \rightarrow \mathbf{k}$ is a ring homomorphism, it follows that $\varepsilon(u)=0$ or $\varepsilon(u)=1$ for each idempotent $u$ of $\mathbf{k}[X]$.
\par
 A non-zero element $u \in \mathbf{k}[X]$ is called a {\it right zero-divisor} if there exists a non-zero element $v \in \mathbf{k}[X]$ such that $vu=0$. Left and two-sided zero-divisors are defined analogously. 
\par
It is a well-known result of Swan \cite[p.571]{MR138688} that if $G$ is a finite group, then the group ring $\mathbf{k}[G]$ has a non-trivial idempotent if and only if some prime divisor of $|G|$ is invertible in $\mathbf{k}$. Although, we do not have Lagrange's theorem for finite quandles, a partial one way analogue of this result does hold for finite quandles.
 
\begin{pro}\label{prop swan analogue}
Let $X$ be a finite quandle having a subquandle $Y$ with more than one element such that $|Y|$ is invertible in $\mathbf{k}$. Then $\mathbf{k}[X]$ has a non-trivial idempotent.
\end{pro}

\begin{proof}
Since the subquandle $Y$ has more than one element, a direct check shows that the element $u= \frac{1}{|Y|}\sum_{y \in Y} e_y$ is a non-trivial idempotent of $\mathbf{k}[X]$. 
\end{proof}
 
\begin{rmk}
The converse of Proposition \ref{prop swan analogue} does not hold in general. For example, consider the quandle $X= \{1, 2, 3 \}$ given in terms of its multiplication table as follows:
\begin{center}
\noindent\begin{tabular}{c | c c c}
	*& 1 & 2 & 3  \\
	\cline{1-4}
	1 & 1 & 1 & 2  \\
	2 & 2 & 2 & 1 \\
	3 & 3 & 3 & 3 \\
\end{tabular}\\
\end{center}
Here, $(i, j)$-th entry of the matrix represents the element $i*j$. The quandle ring $\mathbb{Z}[X]$ has non-trivial idempotents of the form $\alpha e_1+ (1-\alpha) e_2$ for $\alpha \in \mathbb{Z}$, but $X$ has no subquandle $Y$ with more than one element such that $|Y|$ is invertible in $\mathbb{Z}$.
\end{rmk}
  
Throughout the rest of the paper, we assume that if $X$ is a finite quandle, then the order $|X|$ of the quandle is not invertible in the coefficient ring $\mathbf{k}$.  Next, we give a basic sufficient condition that guarantees the existence of non-trivial idempotents.  

\begin{pro}\label{trivial subquandle idempotent}
Let $X$ be a quandle containing a  trivial subquandle $Y$ of order more than one. Then $\mathbf{k}[X]$ has non-trivial idempotents.
\end{pro}

\begin{proof}
Consider the element $u= \sum_{i=1}^n \alpha_i e_{y_i}$, where $n \ge2$,  $y_i \in Y$ and $\alpha_i \in \mathbf{k}$ such that  $\sum_{i=1}^n \alpha_i=1$. A direct check shows that $u^2=u$, and hence $u$ is a non-trivial idempotent of $\mathbf{k}[X]$.
\end{proof}

\begin{lem}\label{faithful}
Let $X$ be a faithful quandle.  If $x, y \in X$ be two distinct elements such that $x * y=x$, then $y * x=y$.
\end{lem}

\begin{proof}
Since $S_yS_x=S_{x * y} S_y$ and $x * y=x$, it follows that $S_x$ and $S_y$ commute.  Thus, the identity $S_x S_y=S_{y * x} S_x$ implies that $S_{y * x}=S_y$.  Since $X$ is faithful, we obtain $y * x=y$, which is desired.
\end{proof}

\begin{pro}\label{fixed point idempotent}
Let $X$ be a faithful quandle such that $S_x$ has more than one fixed-point for some $x \in X$. Then $\mathbf{k}[X]$ has non-trivial idempotents.
\end{pro}

\begin{proof}
Since $S_x$ has a non-trivial fixed-point, we have $y*x=y$ for some $y \in X$ with $y \ne x$. But, $X$ is faithful, and hence by Lemma \ref{faithful}, we have $x * y=x$. The result now follows from Proposition \ref{trivial subquandle idempotent}.
\end{proof}

\begin{pro}
If $G$ is a non-trivial group, then $\mathbf{k}[\Conj(G)]$ has non-trivial idempotents.
\end{pro}

\begin{proof}
Note that, for each non-identity element $x \in G$ and distinct integers $i, j$, the set $ \{x^i, x^j \}$ forms a trivial subquandle of $\Conj(G)$. The result now follows from Proposition \ref{trivial subquandle idempotent}.
\end{proof}

As an application to quandle rings of link quandles, we have

\begin{pro}\label{Hopf link idempotent}
Let $L$ be a link containing the Hopf link and $Q(L)$ the corresponding link quandle of $L$. Then $\mathbf{k}[Q(L)]$ has non-trivial idempotents.
\end{pro}

\begin{proof}
Let $H$ be the Hopf link. It follows from the construction of the link quandle \cite{MR2628474, MR0672410} that $Q(L)$ contains $Q(H)$ as a subquandle with two elements. The result now follows from Proposition \ref{trivial subquandle idempotent}.
\end{proof}

It has been speculated in \cite{BPS1} that connected quandles have only trivial idempotents. We give two examples showing that this is not true in general.

\begin{ex}\label{involutive and faithful example1}
Let $X= \{1, 2, \ldots, 6 \}$ be the connected quandle of order 6 (see \cite{rig}) given in terms of its right multiplications as follows:
\begin{eqnarray*}
&& S_1=(3\;5)(4\;6), \quad S_2=(3\;6)(4\;5), \quad S_3=(1\;5)(2\;6),\\
&& S_4=(1\;6)(2\;5), \quad S_5=(1\;3)(2\;4), \quad S_6=(1\;4)(2\;3).
\end{eqnarray*}
Note that $X$ has trivial subquandles $\{1,2\}$, $\{3, 4\}$ and $\{5, 6\}$. Hence, by Proposition \ref{trivial subquandle idempotent}, the elements
$\alpha e_1+(1-\alpha) e_2$, $\beta e_3 +(1-\beta)e_4$ and $\gamma e_5+(1-\gamma)e_6$  are non-trivial idempotents of $\mathbb{Z}[X]$ for any $\alpha, \beta, \gamma \in \mathbb{Z}$.
\end{ex}

\begin{ex}\label{involutive and faithful example2}
Consider the connected quandle $X=\{1, 2, \ldots,12\}$ of order $12$ (see \cite{rig}) given in terms of its right multiplications as follows:
\begin{eqnarray*}
&& S_1=(5\; 11 \;7\;9)(6\;12\;8\;10), \quad S_2=(5\; 12 \;7\;10)(6\;11\;8\;9), \quad S_3=(5\; 9 \;7\;11)(6\;10\;8\;12),\\
&& S_4=(5\; 10 \;7\;12)(6\;9\;8\;11), \quad S_5=(1\; 9 \;3\;11)(2\;10\;4\;12), \quad S_6=(1\; 10 \;3\;12)(2\;9\;4\;11),\\
&& S_7=(1\; 11 \;3\;9)(2\;12\;4\;10), \quad S_8=(1\; 12 \;3\;10)(2\;11\;4\;9), \quad S_9=(1\; 7 \;3\;5)(2\;8\;4\;6),\\
&& S_{10}=(1\; 8 \;3\;6)(2\;7\;4\;5), \quad \quad S_{11}=(1\; 5 \;3\;7)(2\;6\;4\;8), \quad \quad S_{12}=(1\; 6 \;3\;8)(2\;5\;4\;7). 
\end{eqnarray*}
We see that $X$ has trivial subquandles $\{1,2,3,4\}$, $\{5,6,7,8\}$ and $\{9,10,11,12\}$. By Proposition \ref{trivial subquandle idempotent}, the elements $\alpha e_1+\beta e_2+\gamma e_3 +(1-\alpha-\beta-\gamma)e_4$, $\alpha e_5+\beta e_6+\gamma e_7 +(1-\alpha-\beta-\gamma)e_8$ and $\alpha e_9+\beta e_{10}+\gamma e_{11} +(1-\alpha-\beta-\gamma)e_{12}$ are non-trivial idempotents of $\mathbb{Z}[X]$ for any $\alpha, \beta, \gamma \in \mathbb{Z}$.
\end{ex}

A computer assisted check \cite{Maple15} for quandles of order less than seven suggests the following.

\begin{conj}\label{Main conjecture}
The integral quandle ring of a semi-latin quandle has only trivial idempotents. In particular, the integral quandle ring of a finite latin quandle has only trivial idempotents.
\end{conj}

It follows from \cite[Theorem 3.5]{BPS1} that free quandles are semi-latin. We shall prove in Theorem \ref{idempotents in free products} that free quandles satisfy Conjecture \ref{Main conjecture}. Further, the next result gives more supporting evidence to the conjecture.

\begin{pro}\label{no three elements prop}
Let $G$ be an abelian group without 2 and 3-torsion. Then $\mathbb{Z}[\Core(G)]$ has no non-trivial idempotent built up with at most three distinct basis elements.
\end{pro}

\begin{proof}
Let $u=\alpha e_x + \beta e_y + \gamma e_z$ be an idempotent of $\mathbb{Z}[\Core(G)]$, where $x, y, z \in G$ are distinct and $\alpha, \beta, \gamma \in \mathbb{Z}$.  We have the following two cases:
\par
Case 1: Suppose that precisely two of $\alpha$, $\beta$ and $\gamma$ are non-zero. Without loss of generality, we can take $u=\alpha e_x + \beta e_y$ with $\alpha \ne 0$ and $\beta \ne 0$. Then 
$u=u^2$ gives 
\begin{small}
\begin{equation}\label{eq1}
\alpha e_x + \beta e_y =\alpha^2 e_x + \beta^2 e_y +  \alpha\beta e_{2x-y} + \alpha\beta e_{2y-x}.
\end{equation}
\end{small}
Clearly, $e_{2x-y} \neq e_x$ and $e_{2y-x} \neq e_{y}$. Further, $e_{2x-y} \neq e_{y}$ and $e_{2y-x} \neq e_x$ since $G$ has no 2-torsion. This gives $\alpha\beta=0$, which is a contradiction. Hence, this case does not arise.
\par
Case 2: Suppose that all of  $\alpha$, $\beta$ and $\gamma$ are non-zero.  Then 
$u=u^2$ gives
\begin{small}
\begin{equation}\label{eq2}
\alpha e_x + \beta e_y + \gamma e_z=\alpha^2 e_x + \beta^2 e_y + \gamma^2 e_z +  \alpha\beta e_{2x-y} + \alpha\beta e_{2y-x} + \beta \gamma e_{2y-z} + \beta\gamma e_{2z-y}+ \alpha\gamma e_{2z-x} + \alpha\gamma e_{2x-z}. 
\end{equation}
\end{small}
Note that $e_{2y-z} \ne e_{2z-y}$, $e_{2y-x} \ne e_{2x-y}$ and $e_{2z-x} \ne e_{2x-z}$ since $G$ has no 3-torsion. Clearly, $e_x \ne e_y, e_z,  e_{2x-y}, e_{2x-z}$. Further,  $e_x \ne e_{2y-x}, e_{2z-x}$ since $G$ has no 2-torsion. Thus, $e_x$ equals to at most one of $e_{2y-z}$ and $e_{2z-y}$. Similarly, $e_y$ coincides with at most one of $e_{2x-z}$ and $e_{2z-x}$. And, $e_z$ equals to at most one of $e_{2x-y}$ and $e_{2y-x}$. We now compare coefficients of $e_x$ on both the sides of \eqref{eq2}. 
\par
Case 2(a): If $e_x= e_{2y-z}$, then
\begin{equation}\label{eq3}
\alpha= \alpha^2+\beta \gamma.
\end{equation}
But, $e_x= e_{2y-z}$ also implies that $e_z= e_{2y-x}$. Comparing coefficients of $e_z$ on both the sides of \eqref{eq2} gives
\begin{equation}\label{eq4}
\gamma= \gamma^2 +\alpha\beta. 
\end{equation}
Adding \eqref{eq3} and  \eqref{eq4} gives $\alpha+ \gamma=\alpha^2+ \gamma^2 + \beta(\alpha+ \gamma)$. Now we compare coefficients of $e_y$ on both the sides of \eqref{eq2}. If $e_y$ appears only once on the right hand side of \eqref{eq2}, then $\beta=\beta^2$, and hence $\beta=1$. This gives $\alpha^2+ \gamma^2=0$. Since the coefficients are from $\mathbb{Z}$, this implies that $\alpha=\gamma=0$, which is a contradiction. If $e_y= e_{2x-z}$, then $e_z= e_{2x-y}$, a contradiction. Similarly, if $e_y= e_{2z-x}$, then $e_x= e_{2z-y}$, which is again a contradiction. Hence Case 2(a) does not arise.
\par
Case 2(b): If $e_x= e_{2z-y}$, then proceeding as above, we see that this case does not arise.
\par
Hence, it follows that $e_x$ appears on the right hand side of \eqref{eq2} precisely once. Thus, $\alpha=\alpha^2$, and consequently $\alpha=1$. Repeating the process for $e_y$ and $e_z$, we obtain $\beta=1$ and $\gamma=1$. But, this gives $\varepsilon(u)=3$, a contradiction. Hence $\mathbb{Z}[\Core(G)]$ has no non-trivial idempotent built up with at most three distinct basis elements.
\end{proof}

\begin{cor}
If $n$ is coprime to 2 and 3, then $\mathbb{Z}[R_n]$ has no non-trivial idempotent built up with at most three distinct basis elements.
\end{cor}

\begin{rmk}\label{quasi group example}
Given a non-empty set $X$ and a ring $\mathbf{k}$, let  $\mathbf{k}[X]$ be the free  $\mathbf{k}$-module on the set $X$. Then a binary operation on $X$ can be used to define a ring structure on $\mathbf{k}[X]$ by imitating the construction of a quandle or a group ring. An {\it idempotent quasigroup} is a set $X$ with a binary operation such that both left and right multiplications by elements of $X$ are bijections of $X$ and $x*x=x$ for all $x \in X$. It is worth mentioning that Conjecture \ref{Main conjecture} does not hold if we replace latin quandles by idempotent quasigroups.  As a counterexample, consider the  idempotent quasigroup with multiplication table as follows:
\begin{center}
\noindent\begin{tabular}{c | c c c c c c c c}
	*& 1 & 2 & 3 & 4 & 5 & 6& 7 & 8  \\
	\cline{1-9}
	1 & 1 & 3 & 2 & 5 & 6 & 4& 8 & 7 \\
	2 & 5 & 2 & 1 & 7 & 8 & 3& 4 & 6 \\
	3 & 4 & 6 & 3 & 8 & 7 & 1& 5 & 2 \\
	4 & 6 & 8 & 7 & 4 & 3 & 5& 2 & 1 \\
	5 & 8 & 7 & 4 & 6 & 5& 2& 1 & 3  \\
	6 & 7 & 4 & 8 & 2 & 1& 6& 3 & 5  \\
	7 & 3 & 5 & 6 & 1 & 2 & 8& 7 & 4 \\
	8 & 2 & 1 & 5 & 3 & 4 & 7& 6 & 8 
\end{tabular}\\
\end{center}

A direct computation shows that $u=e_2-e_3-e_6+e_7$ is an idempotent of the ring $\mathbb{Z}[X]$.  This suggests that a proof of Conjecture \ref{Main conjecture} should use the right-distributivity of the quandle in an essential way.
\end{rmk}

The following result is interesting in its own.

\begin{pro}\label{medial endomorphisms}
Let $X$ be a medial quandle. Then the following hold:
\begin{enumerate}
\item The right multiplication by an idempotent is a ring endomorphism of  $\mathbf{k}[X]$.
\item If $X$ is finite, then right multiplications by distinct idempotents give distinct ring endomorphisms of  $\mathbf{k}[X]$.
\end{enumerate}
\end{pro}

\begin{proof}
Let $u= \sum_{i=1}^n \alpha_i e_i$ be an idempotent of $\mathbf{k}[X]$. Let $\hat{S}_u: \mathbf{k}[X] \to \mathbf{k}[X]$ be the map given by
$\hat{S}_u (w)= wu$ for all $w \in  \mathbf{k}[X]$. Let $e_k, e_l$ be two basis elements of $\mathbf{k}[X]$. Then, we see that
\begin{eqnarray*}
\hat{S}_u (e_k e_l) &=&  e_{k* l}~  \big(\sum_{i, j=1}^n \alpha_i \alpha_j e_{i*j} \big),~\quad \textrm{since}~ u=u^2\\
&=& \sum_{i, j=1 }^n \alpha_i \alpha_j  e_{(k *l)* (i *j)}\\
&=& \sum_{i, j=1 }^n \alpha_i \alpha_j  e_{(k *i)(l *j)},~\quad \textrm{since}~ X~ \textrm{is medial}\\
&=& \sum_{i, j=1 }^n \alpha_i \alpha_j  e_{k *i}e_{l *j}\\
&=& \big(\sum_{i=1 }^n\alpha_i e_{k *i}\big) \big(\sum_{j=1 }^n \alpha_j e_{l *j}\big)\\
&=& \hat{S}_u (e_k)~ \hat{S}_u( e_l).
\end{eqnarray*}
Since $\hat{S}_u$ is $\mathbf{k}$-linear, it is a ring homomorphism of $\mathbf{k}[X]$, which proves (1).
\par
For assertion (2), suppose that $X$ is finite of order $n$. Let $u= \sum_{i=1}^n \alpha_i e_i$ and $v= \sum_{i=1}^n \beta_i e_i$ be two idempotents of $\mathbf{k}[X]$. If $\hat{S}_u= \hat{S}_v$, then $$ \sum_{i=1}^n \alpha_i e_{k*i}=\hat{S}_u(e_k)=\hat{S}_v(e_k)=\sum_{i=1}^n \beta_i e_{k*i}$$ for any basis element $e_k$. This gives $\alpha_i=\beta_i$ for all $i$, which shows that $u=v$.
\end{proof}

\begin{rmk}\label{failure injectivity right mult}
Consider the quandle $X$ of Example \ref{involutive and faithful example1}. Take  $\mathbf{k}= \mathbb{Q}$, $u=\frac{1}{2}(e_1+e_2)$ and $\hat{S}_u$ the right multiplication by $u$. Then $u$ is an idempotent of $\mathbb{Q}[X]$ and  $e_3-e_4  \in \ker(\hat{S}_u)$. Thus, the $\mathbf{k}$-linear map $\hat{S}_u$ need not be injective in general.
\end{rmk}

\begin{rmk}\label{failure of distributivity}
The set of idempotents of a quandle ring fails to satisfy right distributivity in general. For example, consider the quandle ring  $\mathbb{Z}[X]$ of the quandle $X$ of Example \ref{involutive and faithful example1}. Take the idempotents $u=e_1$, $v=e_4$ and $w=\alpha e_5 + (1- \alpha) e_6$ for $\alpha \ne 0, 1$. Then a direct check shows that $(uv)w=e_6$, whereas $(uw)(vw)=(2 \alpha- 2\alpha^2) e_5+ (2 \alpha^2-2 \alpha +1) e_6$.
\end{rmk}

\begin{rmk}
Let $X$ be a non-trivial quandle and $\mathbf{k}$ a field. On contrary to associative algebras, the right multiplication $\hat{S}_u$ by an idempotent $u$ of $\mathbf{k}[X]$ is not a projection of the underlying $\mathbf{k}$-vector space $\mathbf{k}[X]$. Thus, the spectrum of the idempotent $u$ (defined as the spectrum of the $\mathbf{k}$-linear map $\hat{S}_u$) may be arbitrary. See \cite[Section 7]{ENS2022} for some related results.
\end{rmk}

\medskip

\section{Idempotents and quandle coverings}\label{coverings and idempotents}
In this section, we use the idea of a quandle covering for giving a precise description of idempotents in quandle rings of quandles of finite type. The notion of a quandle covering is attributed to the work of Eisermann \cite{MR1954330, MR3205568}. 
\par

A quandle homomorphism $p: X  \to Y$ is called a {\it quandle covering} if it is surjective and $S_x=S_{x'}$ whenever $p(x) = p(x')$ for any $x, x' \in X$. Clearly, an isomorphism of quandles is a quandle covering, called a {\it trivial covering}. 

\begin{ex}
Some basic examples of quandle coverings are:
\begin{enumerate}
\item A surjective group homomorphism $p: G \to H$ yields a quandle covering $\Conj(G) \to \Conj(H)$ if and only if $\ker(p)$ is a central subgroup of $G$.
\item A surjective group homomorphism $p: G\to H$ yields a quandle covering $\Core(G) \to \Core(H)$ if and only if $\ker(p)$ is a central subgroup of  $G$ of exponent two.
\item Let $X$ be a quandle and  $F$ a non-empty set viewed as a trivial quandle. Consider $X \times F$ with the product quandle structure $(x,s) * (y,t) = (x *y,s)$. Then the projection $p: X \times F \to X$ given by $(x,s) \to x$ is a quandle covering, called trivial covering with fibre $F$.
\item Let $X$ be a quandle and $A$ an abelian group. A map $\alpha: X \times X \to A$ is called a quandle 2-cocycle if it satisfies 
\begin{equation*} \label{group-coefficient-cocycle-condition}
\alpha_{x, y}~\alpha_{x*y, z}= \alpha_{x, z}~\alpha_{x* z, y*z}
\end{equation*}
and 
\begin{equation*}\label{normalised-cocycle-condition}
\alpha_{x, x}=1
\end{equation*}
for $x, y, z \in X$. Given a 2-cocyle  $\alpha$, the set $X \times A$ turns into a quandle with the binary operation
\begin{equation*}\label{dynamical-quandle-operation}
(x, s)* (y,t)= \big( x* y, ~s~\alpha(x, y) \big),
\end{equation*} 
for $x, y \in X$ and $s, t \in A$. The quandle so obtained is called an  {\it extension} of $X$ by $A$ through $\alpha$, and is denoted by $X \times_{\alpha} A$. We refer the reader to \cite{MR1994219} for generalities and related results. A direct check shows that the projection $p:X \times_{\alpha} A \to X$ given by $p(x, s)=x$ is a quandle covering.
\end{enumerate}
\end{ex}

The following lemma summarises some basic properties of quandle coverings, which we shall use without stating explicitly.

\begin{lem}\label{quandle covering properties}
If $p: X  \to Y$ is a quandle covering, then the following hold:
\begin{enumerate}
\item Each fibre $p^{-1}(y)$ is a trivial subquandle of $X$. 
\item Each inner automorphism of $X$ permutes fibres.
\item The fibres over any two elements of the same connected component of $Y$ are isomorphic.
\end{enumerate}
\end{lem}

\begin{proof}
If $p: X  \to Y$ is a quandle homomorphism, then each fibre $p^{-1}(y)$ is a subquandle of $X$. Since $p$ is a covering, $S_x=S_{x'}$ whenever $x, x' \in p^{-1}(y)$. This gives $x*x'=S_{x'}(x)=S_x(x)=x$ and $x'*x=S_{x}(x')=S_{x'}(x')=x'$, which proves assertion (1).
\par
For assertion (2), it is enough to check that if $x_1, x_2 \in p^{-1}(y)$, then $S_x(x_1)$ and $S_x(x_2)$ are in the same fibre. Indeed, $p(S_x(x_1))=p(x_1*x)=
y*p(x)=p(x_2*x)=p(S_x(x_2))$, and we are done.
\par
Let $y, y'$ be elements of the same connected  component of $Y$. Then there exists elements $y_1, y_2, \ldots, y_n \in Y$ and $\mu_1, \mu_2, \ldots, \mu_n \in \{1, -1 \}$ such that $y'=y *^{\mu_1}y_1*^{\mu_2}y_2 \cdots *^{\mu_n}y_n$. Here the parentheses are left normalised. For each $i$,  choose one element $x_i \in p^{-1}(y_i)$. If $x \in p^{-1}(y)$, then we see that 
$$p(x*^{\mu_1} x_1*^{\mu_2} x_2 \cdots *^{\mu_n} x_n)= y*^{\mu_1}y_1*^{\mu_2}y_2 \cdots *^{\mu_n}y_n= y'.$$
Thus, the inner automorphism $S_{x_n}^{\mu_n}S_{x_{n-1}}^{\mu_{n-1}} \cdots S_{x_1}^{\mu_1}$ maps the fibre $p^{-1}(y)$ bijectively onto  $p^{-1}(y')$, which proves (3).
\end{proof}

\begin{pro}\label{covering with non-trivial idempotents}
If $p:X \to Y$ is a non-trivial quandle covering, then $\mathbf{k}[X]$ has non-trivial idempotents.
\end{pro}

\begin{proof}
Since $p$ is a non-trivial covering, there is at least one connected component of $Y$ such that $|p^{-1}(y)| \ge 2$ for all elements $y$ of that connected component. By assertion (1) of Lemma \ref{quandle covering properties}, $p^{-1}(y)$ is a trivial subquandle of $X$. The result now follows from Proposition \ref{trivial subquandle idempotent}.
\end{proof}

A long knot $L$ is the image of a smooth embedding $\ell: \mathbb{R} \hookrightarrow \mathbb{R}^3$ such that $\ell(t) = (t, 0, 0)$ for all $t$ outside some compact interval. We consider long knots only up to isotopy with compact support. The closure of a long knot is a usual knot defined in the obvious way.

\begin{pro}\label{idempotents long knots}
If $L$ is a non-trivial long knot, then the quandle ring $\mathbf{k}[Q(L)]$ of its knot quandle $Q(L)$ has non-trivial idempotents.
 \end{pro}
 
\begin{proof}
Let $L$ be a long knot and $K$  its corresponding closed knot defined in the obvious way. Let $Q(L)$ and $Q(K)$ be knot quandles of $L$ and $K$, respectively. Note that $Q(K)$ is obtained from $Q(L)$ by adjoining one extra relation corresponding to the first and the last arc of $L$. By \cite[Theorem 35]{MR1954330}, the natural projection $p: Q(L) \to Q(K)$ is a non-trivial quandle covering, and the result follows from Proposition \ref{covering with non-trivial idempotents}.
 \end{proof}
 
Let $p:X \to Y$ be a quandle covering, and $\mathcal{F}(Y)$ the set of all finite subsets of $Y$. For each $y \in Y$, let $\mathcal{F}(p^{-1}(y))$ be the set of all finite subsets of $p^{-1}(y)$, and denote a typical element of this set by $I_y$. Given elements $x,y$ in a quandle $X$ of finite type, we set
$$[e_x]_y:= e_x+ e_{x*y}+ e_{x*y*y}+ \cdots +e_{x*\underbrace{y*y*\cdots *y}_{(n_y-1)-\textrm{times}}}=e_x+ e_{S_y(x)}+ e_{S_y^2(x)}+ \cdots +e_{S_y^{n_y-1}(x)},$$
the sum of the basis elements in the $S_y$-orbit of $e_x$. The main result of this section is the following theorem.

\begin{thm}\label{quandle covering theorem}
Let $X$ be a quandle of finite type and $p:X \to Y$ a non-trivial quandle covering. If $\mathbf{k}[Y]$ has only trivial idempotents, then the set of idempotents of $\mathbf{k}[X]$ is
\begin{small}
\begin{eqnarray}\label{quandle covering theorem idempotents}
\nonumber \mathcal{I} \big(\mathbf{k}[X]\big) &=& \Big\{ \sum_{y \in J} \Big( \sum_{x \in I_y,~~\sum \alpha_x =0}  \alpha_x~ [e_x]_{x_0} \Big) + \Big(\sum_{x' \in I_{y_0},~~\sum \alpha_{x'} =1} \alpha_{x'} ~e_{x'}\Big) ~\bigl\vert \\
 && J \in \mathcal{F}(Y),~~I_y \in \mathcal{F}(p^{-1}(y)),~~I_{y_0} \in \mathcal{F}(p^{-1}(y_0)),~~ x_0 \in  I_{y_0}, ~~y_0 \in Y, ~~\alpha_x, \alpha_{x'}  \in \mathbf{k} \Big\}.
\end{eqnarray}
\end{small}
\end{thm}
 
\begin{proof}
Since $p$ is a quandle covering, we have $S_x=S_{x'}$ for any $x, x ' \in p^{-1}(y)$. Hence the induced automorphisms of the quandle ring $\mathbf{k}[X]$ are identical for any $x, x ' \in p^{-1}(y)$. This together with direct computations give
\begin{eqnarray}\label{even dihedral eq1}
&&\Big(\sum_{x \in J} \beta_{x}~ e_{x}\Big)\Big(\sum_{x' \in I_y,~~\sum \alpha_{x'} =1} \alpha_{x'} ~e_{x'}\Big)\\
\nonumber &=& \sum_{x' \in I_y,~~\sum \alpha_{x'} =1}\alpha_{x'} \Big( \sum_{x \in J} \beta_{x}~ e_{x} \Big) e_{x'}\\
\nonumber &=&\sum_{x' \in I_y,~~\sum \alpha_{x'} =1} \alpha_{x'} \Big( \sum_{x \in J} \beta_{x}~ e_{x} \Big)e_{x_0},\quad \textrm{for any fixed}~x_0\in I_y\\
\nonumber &=&\sum_{x' \in I_y,~~\sum \alpha_{x'} =1} \alpha_{x'} \Big( \sum_{x \in J} \beta_{x} ~e_{x*x_0} \Big)\\
\nonumber  &=&   \sum_{x \in J} \beta_{x}~ e_{x*x_0},\quad \textrm{since}~\sum_{x' \in I_y,~~\sum \alpha_{x'} =1} \alpha_{x'} =1,
\end{eqnarray}
and
\begin{eqnarray}\label{even dihedral eq2}
&&  \Big(\sum_{x \in J} \beta_{x} ~e_{x}\Big) \Big( \sum_{x' \in I_y,~~\sum \alpha_{x'} =0}  \alpha_{x'} ~e_{x'} \Big)\\
\nonumber &=&  \sum_{x' \in I_y,~~\sum \alpha_{x'} =0}  \alpha_{x'} \Big(\sum_{x \in J} \beta_{x} ~e_{x}\Big) e_{x'}\\
\nonumber &=&  \sum_{x' \in I_y,~~\sum \alpha_{x'} =0}  \alpha_{x'}\Big(\sum_{x \in J} \beta_{x}~ e_{x*x_0}\Big), \quad \textrm{for any fixed}~x_0 \in I_y\\
\nonumber &=& \Big(\sum_{x' \in I_y,~~\sum \alpha_{x'} =0}  \alpha_{x'} \Big) \Big(\sum_{x \in K} \beta_{x} ~e_{x*x_0}\Big)\\
\nonumber &=& 0,\quad \textrm{since}~ \sum_{x' \in I_y,~~\sum \alpha_{x'} =0}  \alpha_{x'}=0,
\end{eqnarray}
where $J \in \mathcal{F}(X)$, $I_y \in \mathcal{F}(p^{-1}(y))$,  $y \in Y$ and $\beta_x, \alpha_{x'}  \in \mathbf{k}$. Let $u=v+w$, where 
\begin{eqnarray*}
v &=&\sum_{y \in J} \Big( \sum_{x \in I_y,~~\sum \alpha_x =0}  \alpha_x ~[e_x]_{x_0}\Big),\\
w &=& \sum_{x' \in I_{y_0},~~\sum \alpha_{x'} =1} \alpha_{x'} ~e_{x'},\\
\end{eqnarray*}
$J \in \mathcal{F}(Y)$  and $x_0 \in  I_{y_0}$ a fixed element. Equations \eqref{even dihedral eq1} and \eqref{even dihedral eq2} imply that $w^2=w$, $wv=0$ and  $v^2=0$. By definition of the element $[e_x]_{x_0}$, it follows that $\big([e_x]_{x_0}\big)e_{x_0}=[e_x]_{x_0}$. Consequently, $vw=v$, and hence $u^2=u$. 
\par
For the converse, let $u$ be a non-zero idempotent of $\mathbf{k}[X]$. Since $X$ is the disjoint union of fibres of $p$, we can write $u$ uniquely in the form
$$u=\sum_{y \in J} \Big( \sum_{x \in I_y}  \alpha_x ~e_x \Big)$$
for some $J \in \mathcal{F}(Y)$ and $I_y \in \mathcal{F}(p^{-1}(y))$ for each $y \in J$. If $\hat{p}: \mathbf{k}[X] \to \mathbf{k}[Y]$ is the induced homomorphism of rings, then $\hat{p}(u)$ is an idempotent of $\mathbf{k}[Y]$. It follows from the decomposition of $u$ that
$$\hat{p}(u)=  \sum_{y \in J} \Big( \sum_{x \in I_y}  \alpha_x \Big) e_y.$$
Since $\mathbf{k}[Y]$ has only trivial idempotents, it follows that either $\hat{p}(u)=0$ or precisely one of the coefficients of $\hat{p}(u)$ is 1 and all other coefficients are 0. If $\hat{p}(u)=0$, then $\sum_{x \in I_y}  \alpha_x =0$ for each $y \in J$. Writing 
$$u=\sum_{y \in J} \Big( \sum_{x \in I_y, ~~\sum \alpha_{x} =0}  \alpha_x ~e_x \Big),$$
 it follows from \eqref{even dihedral eq2} that $u=u^2=0$, which is a contradiction as $u \ne 0$. Hence, there exists $y_0 \in J$ such that $\sum_{x' \in I_{y_0}}  \alpha_{x'}=1$ and $\sum_{x \in I_y}  \alpha_x=0$ for all $y \neq y_0$. We can write $u= v+w$, where 
$$v= \sum_{y \in J, ~y \neq y_0} \Big( \sum_{x \in I_y,~~\sum \alpha_x =0}  \alpha_x~ e_x \Big)$$ and $$w= \sum_{x' \in I_{y_0},~~\sum \alpha_{x'} =1} \alpha_{x'}~ e_{x'}.$$ Again, equations \eqref{even dihedral eq1} and \eqref{even dihedral eq2} imply that $w^2=w$, $wv=0$ and  $v^2=0$. Thus, we have
$$u=u^2=v^2+w^2+vw+wv= w+vw,$$
and consequently $v=vw$. This implies that 
$$\sum_{y \in J, ~y \neq y_0} \Big( \sum_{x \in I_y,~~\sum \alpha_x =0}  \alpha_x~ e_x \Big)= \sum_{y \in J, ~y \neq y_0} \Big( \sum_{x \in I_y,~~\sum \alpha_{x} =0}  \alpha_{x}~ e_{x} \Big) e_{x_0}.$$
for some fixed $x_0 \in I_{y_0}$. Hence, it follows that $v$ has the form
$$v=\sum_{y \in J', ~y \neq y_0} \Big( \sum_{x \in I_y,~~\sum \alpha_x =0}  \alpha_x ~[e_x]_{x_0} \Big),$$
where $J' \subseteq J$ is a set of representatives of orbits of the action of $S_y$ on $J$. This completes the proof of the theorem.
 \end{proof}
 
 \begin{cor}
If $X$ is a trivial quandle, then 
$$\mathcal{I}\big(\mathbf{k}[X]\big) =\Big\{ \sum_{x \in J} \alpha_x ~e_x ~\mid~   J \in \mathcal{F}(X)~\textrm{such that}~ \sum_{x \in J} \alpha_x=1 \Big\}.$$
 \end{cor}
 \begin{proof}
If $\{z \}$ is a one element quandle, then the constant map $c: X \to \{z \}$ is a quandle covering. The proof now follows from Theorem \ref{quandle covering theorem}.
 \end{proof}

\begin{cor}
Let $p:X \to Y$ be a non-trivial quandle covering such that $\mathbf{k}[Y]$ has only trivial idempotents. Then every idempotent of $\mathbf{k}[X]$ has augmentation value 1.
\end{cor}

\begin{proof}
The assertion follows from the proof of the converse part of Theorem \ref{quandle covering theorem}. Note that we do not need our quandles to be of finite type.
\end{proof}

It has been shown in \cite{BPS1} that the integral quandle ring of $R_3$ has only trivial idempotents. A computer assisted check shows that the same assertion holds for the integral quandle ring of $R_5$ as well \cite[Section 6]{ENS2022}. As an application of the preceding theorem, we characterise idempotents in quandle rings of certain dihedral quandles of even order under the assumption of Conjecture \ref{Main conjecture}.
\begin{cor}\label{even dihedral corollary}
Let $n=2m+1$ be an odd integer with $m \ge 1$. Assume that $\mathbf{k}[R_{n}]$ has only trivial idempotents. Then the set of idempotents of $\mathbf{k}[R_{2n}]$ is given by
\begin{small}
$$\mathcal{I}\big(\mathbf{k}[R_{2n}]\big) =\Big\{ \big(\beta e_j + (1-\beta)e_{n+j}\big)+ \sum_{i=0}^{m} \alpha_i \big(e_i-e_{n+i} + e_{2j-i}- e_{n+2j-i}\big) ~\bigl\vert ~ 0 \le j \le n-1 ~\textrm{and~} ~~ \alpha_i, \beta \in \mathbf{k} \Big\}.$$
\end{small}
 \end{cor}
 
 \begin{proof}
Note that the natural map $p:R_{2n} \to R_{n}$ given by reduction modulo $n$ is a two-fold non-trivial quandle covering. Further, for each $i \in R_n$,  we have $p^{-1}(i)=\{i, n+i\}$. The result now follows from Theorem \ref{quandle covering theorem}. 
 \end{proof}

\begin{pro}\label{zero-divisors for coverings}
Let $p:X \to Y$ be a non-trivial quandle covering. Then $\mathbf{k}[X]$ has right zero-divisors.
 \end{pro}
 
\begin{proof}
Let $J \in \mathcal{F}(X)$, $y \in Y$ and $I_y \in \mathcal{F}(p^{-1}(y))$ such that $|I_y| \ge 2$. Then for any $\sum_{x \in I_y,~~\sum \alpha_x =0}  \alpha_x e_x$ and $\sum_{x' \in J} \beta_{x'} e_{x'}$, it follows from \eqref{even dihedral eq2} that
$$\Big(\sum_{x' \in J} \beta_{x'} ~e_{x'}\Big) \Big( \sum_{x \in I_y,~~\sum \alpha_x =0}  \alpha_x ~e_x \Big)=0$$
and hence  $\sum_{x \in I_y,~~\sum \alpha_x =0}  \alpha_x e_x$ is a right zero-divisor of $k[X]$.
 \end{proof}

\begin{pro}
Let $X$ be an involutory quandle such that $\mathbf{k}[X]$ has only trivial idempotents. Let $A$ be a non-trivial abelian group and  $\alpha: X \times X \to A$ a quandle 2-cocycle satisfying $\alpha_{x*y, y}=\alpha_{x, y}^{-1}$ for all $x, y \in X$. Then the extension $X \times_{\alpha} A$ is involutory and $\mathbf{k}[X \times_{\alpha} A]$ has non-trivial idempotents.
\end{pro}

\begin{proof}
A direct check shows that the condition $\alpha_{x*y, y}=\alpha_{x, y}^{-1}$ is equivalent to $X \times_{\alpha} A $ being involutory. Since the map $p:X \times_{\alpha} A \to X$ is a non-trivial quandle covering, the result follows from Theorem \ref{quandle covering theorem}. In fact, Theorem \ref{quandle covering theorem} gives the precise set of idempotents. 
\end{proof}

\begin{pro}\label{idempotents forming quandle covering}
Let $X$ be a quandle of finite type and $p : X \to Y$ a quandle covering. Then the set 
\begin{eqnarray}
\nonumber I &=& \Big\{ \sum_{y \in J} \Big( \sum_{x \in I_y,~~\sum \alpha_x =0}  \alpha_x ~[e_x]_{x_0} \Big) + \Big(\sum_{x' \in I_{y_0},~~\sum \alpha_{x'} =1} \alpha_{x'}~ e_{x'}\Big) ~\bigl\vert \\
\nonumber && J \in \mathcal{F}(Y),~~I_y \in \mathcal{F}(p^{-1}(y)),~~I_{y_0} \in \mathcal{F}(p^{-1}(y_0)),~~ x_0 \in  I_{y_0}, ~~y_0 \in Y,~~\alpha_x, \alpha_{x'}  \in \mathbf{k}\Big\}
\end{eqnarray}
of idempotents of $\mathbf{k}[X]$ is a quandle with respect to the ring multiplication.
\end{pro}

\begin{proof}
Consider the elements 
\begin{eqnarray*}
u &=& \sum_{y \in J_1} \Big( \sum_{x \in I_y,~~\sum \alpha_x =0}  \alpha_x ~ [e_x]_{x_1}  \Big) + \Big(\sum_{x' \in I_{y_1},~~\sum \alpha_{x'} =1} \alpha_{x'}~ e_{x'}\Big),\\
v &=&\sum_{y \in J_2} \Big( \sum_{x \in I_y,~~\sum \beta_x =0}  \beta_x ~[e_x]_{x_2}  \Big) + \Big(\sum_{x' \in I_{y_2},~~\sum \beta_{x'} =1} \beta_{x'}~ e_{x'}\Big)
\end{eqnarray*}
in the set $I$, where $J_i \in \mathcal{F}(Y)$ and $I_{y} \in \mathcal{F}(p^{-1}(y))$, $y_i \in Y$ and $x_i \in I_{y_i}$. Then we have
\begin{eqnarray*}
& uv&\\
 &=& \Big(\sum_{y \in J_1} \Big( \sum_{x \in I_y,~~\sum \alpha_x =0}  \alpha_x ~ [e_x]_{x_1} \Big) + \Big(\sum_{x' \in I_{y_1},~~\sum \alpha_{x'} =1} \alpha_{x'}~ e_{x'}\Big) \Big) \\
&& \Big( \sum_{y \in J_2} \Big( \sum_{x \in I_y,~~\sum \beta_x =0}  \beta_x ~[e_x]_{x_2}  \Big) + \Big(\sum_{x' \in I_{y_2},~~\sum \beta_{x'} =1} \beta_{x'}~ e_{x'}\Big)\Big) \\
&=& \Big(\sum_{y \in J_1} \Big( \sum_{x \in I_y,~~\sum \alpha_x =0}  \alpha_x ~[e_x]_{x_1} \Big) + \Big(\sum_{x' \in I_{y_1},~~\sum \alpha_{x'} =1} \alpha_{x'} ~e_{x'}\Big) \Big) \Big(\sum_{x' \in I_{y_2},~~\sum \beta_{x'} =1} \beta_{x'}~ e_{x'}\Big)\Big),\\
&& \quad \textrm{by}~ \eqref{even dihedral eq2} \\
&=& \Big(\sum_{y \in J_1} \Big( \sum_{x \in I_y,~~\sum \alpha_x =0}  \alpha_x ~[e_x]_{x_1} \Big) + \Big(\sum_{x' \in I_{y_1},~~\sum \alpha_{x'} =1} \alpha_{x'} ~e_{x'}\Big) \Big) e_{x_2}, \quad \textrm{by} ~\eqref{even dihedral eq1}\\
&=& \sum_{y \in J_1} \Big( \sum_{x \in I_y,~~\sum \alpha_x =0}  \alpha_x ~\big([e_x]_{x_1}\big)e_{x_2} \Big) + \Big(\sum_{x' \in I_{y_1},~~\sum \alpha_{x'} =1} \alpha_{x'} ~ e_{x'*x_2}\Big),\\
&=& \sum_{y \in J_1} \Big( \sum_{x*x_2 \in I_{y*y_2},~~\sum \alpha_x =0}  \alpha_x~[e_{x*x_2}]_{x_1*x_2} \Big) + \Big(\sum_{x'*x_2 \in I_{y_1*y_2},~~\sum \alpha_{x'} =1} \alpha_{x'} ~e_{x'*x_2}\Big),
\end{eqnarray*}
since $\big([e_x]_{x_1}\big)e_{x_2}=[e_{x*x_2}]_{x_1*x_2}$ due to right-distributivity. Note here that $x_1*x_2 \in I_{y_1*y_2}$. Thus, we have proved that $uv \in I$. The preceding computation also shows that the right multiplication by $v$ is precisely the right multiplication by 
$e_{x_2}$ for any fixed $x_2 \in I_{y_2}$. In other words, the right multiplication by $v$ is the ring automorphism $\hat{S}_{x_2}$ of $\mathbf{k}[X]$. This proves that the set $I$ is a quandle.
\end{proof}
 
An immediate consequence of Proposition \ref{idempotents forming quandle covering} is the following.
 
 \begin{cor}\label{idempotens form quandle}
 Let $X$ be a quandle of finite type and $p : X \to Y$ a quandle covering. Suppose that $\mathbf{k}[Y]$ has only trivial idempotents. Then the following hold:
\begin{enumerate}
\item The set of all idempotents of $\mathbf{k}[X]$ is  a quandle with respect to the ring multiplication.
\item The right multiplication by each idempotent of $\mathbf{k}[X]$ is a ring automorphism induced by some trivial idempotent of $\mathbf{k}[X]$.
\end{enumerate}
 \end{cor}

Note that Proposition \ref{medial endomorphisms} already proves the endomorphism part in assertion (2) of Corollary \ref{idempotens form quandle} for medial quandles.
\medskip

\section{Idempotents in quandle rings of free products}\label{sec idempotents in free products}
The free product of quandles can be defined as follows. Let $X_i = \langle S_i~|~R_i \rangle$ be a collection of $n \ge 2$ quandles given in terms of presentations. Then their {\it free product} $X_1 \star X_2  \star \cdots \star X_n$ is the quandle defined by the presentation
$$
X_1 \star X_2  \star \cdots \star X_n = \langle S_1 \sqcup S_2 \sqcup \cdots \sqcup S_n ~\mid~R_1 \sqcup R_2 \sqcup \cdots \sqcup R_n\rangle.
$$
For example, the free quandle $FQ_n$ of rank $n$ can be seen as 
$$
FQ_n = \langle x_1\rangle \star \langle x_2\rangle \star \cdots \star\langle x_n\rangle,
$$
the free product of $n$ copies of trivial one element quandles $\langle x_i\rangle$.
\par

Recall that, each element of a quandle $X$ has a canonical left-associated expression $x_0*^{\epsilon_1}x_1*^{\epsilon_2}\cdots*^{\epsilon_n}x_n$, where $x_0\neq x_1$, and if $x_i=x_{i+1}$ for any $1 \le i \le n-1$, then $\epsilon_i= \epsilon_{i+1}$.
\par
Lack of associativity in quandles makes it hard to have a normal form for elements in free products of quandles. We overcome this difficulty by defining a length for elements in free products.  Let $X= X_1 \star X_2 \star \cdots \star X_n$ be the free product of $n \ge 2$ quandles. Given an element $w \in X$, we define the length $\ell(w)$ of $w$ as 
\begin{eqnarray*}
\ell(w) &=& \min \Big\{r \mid w ~\textrm{can be written as a canonical left associated product of} ~r\\
&& \quad \quad  \textrm{elements from}~ X_1 \sqcup X_2 \sqcup \cdots \sqcup X_n \Big\}.
\end{eqnarray*}
Notice that each $w \in X$ has a reduced left associated expression attaining the length $\ell(w)$. This can be done by gathering together all the leftmost alphabets in a left associated expression of $w$ that lie in the same component quandle $X_i$, and rename it as a single element of $X_i$. This shows that $\ell(w)=1$ if and only if $w \in X_i$ for some $i$. Equivalently,  $\ell(w) \ge 2$ if and only if $w \in X\setminus (\sqcup_{s=1}^n X_s)$. 
\par
For example, if $x_1, x_2 \in X_i$ and $y_1, y_2 \in X_j$ for $i \ne j$, then $\ell(x_1*x_2)=1$, $\ell(x_1*x_2*^{-1}x_1)=1$, $\ell(x_1*y_1)=2$, $\ell(x_1*y_1*y_2)=3$ and $\ell(x_1*y_1*y_2*x_2)=4$.
\par
Note that, if $X= X_1 \star X_2 \star \cdots \star X_n$, then every $u \in \mathbf{k}[X]$ can be written uniquely in the form 
\begin{equation}\label{elements in rings of free products}
u= u_1+u_2+ \cdots + u_n + v,
\end{equation} 
where each $u_i \in \mathbf{k}[X_i]$,~ $v = \sum_{k=1}^m\gamma_k e_{w_k}$ with each $\ell(w_k) \ge 2$ and $\gamma_k \in \mathbf{k}$.

\begin{pro}
Let $X= X_1 \star X_2 \star \cdots \star X_n$ be the free product of $n$ quandles such that each $\mathbf{k}[X_i]$ has only trivial idempotents. Then any idempotent $u$ of $\mathbf{k}[X]$ can be written uniquely as
$$u=\alpha_1 e_{x_1} + \alpha_2 e_{x_2} + \cdots +\alpha_n e_{x_n} + v,$$
where $x_i \in X_i$, $v=\sum_{k=1}^m\gamma_k e_{w_k}$ with $\ell(w_k) \ge 2$ and $\alpha_i, \gamma_k \in \mathbf{k}$ for all $i$ and $k$.
\end{pro}

\begin{proof}
For each $i$, fix an element $z_i \in X_i$. Then the maps $p_i: X \to X_i$ defined by setting 
$$
p_i(x)= \left\{
\begin{array}{ll}
x \quad \textrm{if}~~ x \in X_i, &  \\
z_i \quad \textrm{if}~~ x \in X_j  ~~\textrm{for}~~j \ne i.
\end{array} \right.
$$
The universal property of free products implies that each $p_i$ is a quandle homomorphism. Let $u= u_1+u_2+ \cdots + u_n + v$ be an idempotent of $\mathbf{k}[X]$, where each $u_i \in \mathbf{k}[X_i]$,~ $v = \sum_{k=1}^m\gamma_k e_{w_k}$ with $\ell(w_k) \ge 2$ and $\gamma_k \in \mathbf{k}$. Then $\hat{p_i}(u)$ is an idempotent in $\mathbf{k}[X_i]$ for each $i$. Since each $\mathbf{k}[X_i]$ has only trivial idempotents and
$$\hat{p_i}(u)= u_i+ \sum_{j \ne i, ~j=1}^n \epsilon(u_j) e_{z_i} + \hat{p_i}(v),$$
it follows that $u_i=\alpha_i e_{x_i}$ for some $x_i \in X_i$ and $\alpha_i \in \mathbf{k}$. Note that if $ \epsilon(u_j) \neq 0$ for any $j \ne i$, then $x_i=z_i$. Thus, $u=\alpha_1 e_{x_1} + \alpha_2 e_{x_2} + \cdots +\alpha_n e_{x_n} + v$, and we are done.
\end{proof}

\begin{lem}\label{free product lemma}
Let $X= X_1 \star X_2 \star \cdots \star X_n$ be the free product of $n$ quandles with $n \ge 2$. Let $u \in \mathbf{k}[X]$ be an idempotent and
$u=u_1+u_2+ \cdots + u_n + v$ be its unique decomposition as in \eqref{elements in rings of free products}. Suppose that 
\begin{enumerate}
\item $\ell(w_k*w_l) \ge 2$ for any $k$ and $l$.
\item $\ell(w_k*x ) \ge 2$ for any $x \in \sqcup_{s=1}^n X_s$ and any $k$.
\end{enumerate}
Then $u_i$ is an idempotent of $\mathbf{k}[X_i]$ for each $i$.
\end{lem}

\begin{proof}
Since $u=u^2$, we have
\begin{equation}\label{elements in rings of free products2}
u_1+u_2+ \cdots + u_n + v= u_1^2+ u_2^2 + \cdots + u_n^2 + v^2 + \sum_{i \neq j, ~i, j=1}^n u_i u_j +  \sum_{i=1}^n u_i v +  \sum_{j=1}^n v u_j.
\end{equation} 
If $v=0$, then \eqref{elements in rings of free products2} takes the form
\begin{equation}\label{elements in rings of free products3}
u_1+u_2+ \cdots + u_n= u_1^2+ u_2^2 + \cdots + u_n^2 + \sum_{i \neq j, ~i, j=1}^n u_i u_j.
\end{equation}
For $1 \le i \ne j \le n$, each basis element of $\mathbf{k}[X]$ appearing in a product $u_i u_j$ corresponds to a quandle element from $X \setminus (\sqcup_{s=1}^n X_s)$. For each $1 \le i \le n$, gathering all the summands on the right hand side of \eqref{elements in rings of free products3} corresponding to elements from the quandle $X_i$ implies that $u_i=u_i^2$, which is desired.
\par
Now suppose that $v \neq 0$. For each $1 \le k, l\le m$, the condition $\ell(w_k*w_l) \ge 2$ implies that the basis element of $\mathbf{k}[X]$ corresponding to the quandle element $w_k*w_l$ does not appear as a summand for any $u_j$. Further, each basis element appearing in a product $u_i v$ corresponds to a quandle element of the form $x*w_k$ for some $x \in X_i$ and some $1 \le k \le m$. But, we have $\ell(x*w_k) \ge 2$ for such elements. Lastly, the condition $\ell(w_k*x ) \ge 2$ for any $x \in \sqcup_{s=1}^n X_s$ also implies that the basis element of $\mathbf{k}[X]$ corresponding to the quandle element $w_k*x$ does not appear as a summand for any $u_j$. For each $1 \le i \le n$, gathering together all the summands on the right hand side of \eqref{elements in rings of free products2} corresponding to elements from the quandle $X_i$ imply that $u_i=u_i^2$, which is desired.
\end{proof}

We are now ready to prove the main result of this section.

\begin{thm}\label{idempotents in free products}
Let $FQ_n$ be the free quandle of rank $n \ge 1$. Then $\mathbb{Z}[FQ_n]$ has only trivial idempotents.
\end{thm}

\begin{proof}
An analogue of the Nielsen--Schreier theorem stating that every subquandle of a free quandle is free has been proved recently in \cite{IKK2019}. Let $FQ_2= \langle x\rangle  \star \langle y \rangle$ be the free quandle of rank two. Then, 
$$FQ_n \cong \langle x\rangle  \star \langle x*y \rangle \star \langle x*y*y \rangle \star \cdots \star \langle x*\underbrace{y*y* \cdots *y}_{(n-1)~\textrm{times}} \rangle$$
and embeds as a subquandle of $FQ_2$ for each $n \ge 3$. Thus, it suffices to prove that $\mathbb{Z}[FQ_2]$ has only trivial idempotents.
\par
Let  $u= \alpha e_x + \beta e_y +v$  be an idempotent of $\mathbb{Z}[FQ_2]$, where $v=\sum_{k=1}^m \gamma_k e_{w_k}$ with $\ell(w_k) \ge 2$ and $\alpha, \beta, \gamma_k \in \mathbb{Z}$. If $v=0$, then  Lemma \ref{free product lemma} implies that $\alpha e_x=  \alpha^2 e_x$ and $ \beta e_y=  \beta^2 e_y$. Hence, either $u=e_x$ or $u=e_y$, and $u$ is a trivial idempotent.
\par
Now, suppose that $v \neq 0$. Note that the first two leftmost alphabets in the reduced left associated expression of each $w_k$ are distinct. We claim that $\gamma_k=1$ for each $k$. This will be achieved by transforming the idempotent $u$ into a new idempotent such that conditions of Lemma \ref{free product lemma} are satisfied.  Fix a $k$ such that $1 \le k \le m$ and write
$$w_k= x_0 *^{\epsilon_1}x_1 *^{\epsilon_2} x_2 *^{\epsilon_3} \cdots  *^{\epsilon_r} x_r,$$
in its reduced left associated expression, where $x_i \in \{ x, y \}$ and $\epsilon_i \in \mathbb{Z}$ for each $i$. Since the expression is reduced, without loss of generality, we can assume that $x_0=x$ and $x_1=y$. Consider the inner automorphism 
$$\phi=S_{x_0}S_{x_0}S_{x_1}^{-\epsilon_1}S_{x_2}^{-\epsilon_2} \cdots S_{x_{r-1}}^{-\epsilon_{r-1}}S_{x_r}^{-\epsilon_r} $$
of $FQ_2$. We analyse the effect of $\phi$ on each summand of $u$. First note that $\phi(w_k)= x_0=x$. Consider any fixed $w_i$ for $i \neq k$ and write $w_i=y_0 *^{\mu_1}y_1 *^{\mu_2} y_2 *^{\mu_3} \cdots  *^{\mu_s} y_s$ in its reduced left associated expression, where $y_t \in \{ x, y \}$ and $\mu_t \in \mathbb{Z}$ for each $t$. We have
$$\phi(w_i)= y_0 *^{\mu_1}y_1 *^{\mu_2} y_2 *^{\mu_3} \cdots  *^{\mu_s} y_s *^{-\epsilon_r}  x_r *^{-\epsilon_{r-1}}  x_{r-1}*^{-\epsilon_{r-2}} \cdots *^{-\epsilon_1}x_1 * x_0 * x_0.$$
Considering the cases $s=r$, $s>r$ and $s<r$, and using the fact that the set of alphabets is $\{x, y \}$, we obtain $\ell(\phi(w_i)) \ge 3$. This clearly implies that $\phi(w_i)*x, \phi(w_i)*y \not\in \{x, y \}$ for any $i \ne k$. Now consider another $w_j$ for $j \neq k$ and $j \neq i$ and write $w_j=z_0 *^{\nu_1}z_1 *^{\nu_2} z_2 *^{\nu_3} \cdots  *^{\nu_l} z_l$ in its reduced left associated expression, where $z_t \in \{ x, y \}$ and $\nu_t \in \mathbb{Z}$ for each $t$. Then Lemma 
\ref{Lem:the canonical left associated form} gives
\begin{eqnarray*}
&&\phi(w_i) *\phi(w_j)\\
&=&\phi(w_i*w_j)\\
&=&\big((y_0 *^{\mu_1}y_1 *^{\mu_2}  \cdots  *^{\mu_s} y_s)(z_0 *^{\nu_1}z_1 *^{\nu_2}  \cdots  *^{\nu_l} z_l) \big)\\
&&*^{-\epsilon_r}  x_r *^{-\epsilon_{r-1}}  x_{r-1}*^{-\epsilon_{r-2}} \cdots *^{-\epsilon_1}x_1 * x_0 * x_0 \\
&=& y_0 *^{\mu_1}y_1 *^{\mu_2}  \cdots  *^{\mu_s} y_s *^{-\nu_l} z_l *^{-\nu_{l-1}} z_{l-1} *^{-\nu_{l-2}} \cdots *^{-\nu_1}z_1 *z_0 *^{\nu_1}z_1 *^{\nu_2}  \cdots  *^{\nu_l} z_l\\
&&*^{-\epsilon_r}  x_r *^{-\epsilon_{r-1}}  x_{r-1}*^{-\epsilon_{r-2}} \cdots *^{-\epsilon_1}x_1 * x_0 * x_0.
\end{eqnarray*}
As before, by comparing $\ell(w_i*w_j)$ and $r$, we obtain $\ell \big(\phi(w_i) *\phi(w_j)\big) \ge 3$. If $\alpha$ and $\beta$ are non-zero, then $\ell(\phi(x)), \ell(\phi(y)) \ge 3$ for the same reason. Thus, the only summand of the idempotent $\phi(u)= \alpha e_{\phi(x)} + \beta e_{\phi(y)} + \sum_{k=1}^m \gamma_k e_{\phi(w_k)}$ that corresponds to an element from $\{x, y \}$ is $\phi(w_k)$, and all the summands  corresponding to $\phi(w_i)$ for $i \ne k$ satisfy the conditions of  Lemma \ref{free product lemma}. Thus, we obtain $\gamma_k e_{\phi(w_k)} =(\gamma_k e_{\phi(w_k)})^2$, and hence $\gamma_k=1$, which proves the claim. On plugging this information back to $u$, we can write $u= \alpha e_x + \beta e_y + \sum_{k=1}^m e_{w_k}$. Since $u$ is an idempotent, we have
\begin{eqnarray*}
\alpha e_x + \beta e_y + \sum_{k=1}^m e_{w_k} &=& \alpha^2 e_x + \beta^2 e_y + \sum_{k, ~l=1}^m e_{w_k *w_l} + \alpha \beta e_{x*y} +\alpha \beta e_{y*x}\\
&&+  \alpha \sum_{k=1}^m e_{x* w_k} + \alpha \sum_{k=1}^m e_{w_k* x} + \beta \sum_{k=1}^m e_{y* w_k} + \beta \sum_{k=1}^m e_{w_k* y}.
\end{eqnarray*}
Comparing coefficients of $e_x$ gives
$$\alpha= \alpha^2, \quad \alpha= \alpha^2 + \sum_{w_k *w_l=x} 1, \quad \alpha= \alpha^2 + \beta \quad \textrm{or} \quad \alpha= \alpha^2 + \sum_{w_k *w_l=x} 1 + \beta.$$
Similarly, comparing coefficients of $e_y$ gives
$$\beta = \beta ^2, \quad \beta = \beta ^2 + \sum_{w_k *w_l=y} 1, \quad \beta = \beta ^2 + \alpha \quad \textrm{or} \quad \beta = \beta ^2 + \sum_{w_k *w_l=y} 1 + \alpha.$$
Using the fact that the coefficients are from $\mathbb{Z}$, a direct check shows that the only possible cases are
\begin{eqnarray*}
\alpha= \alpha^2 &\textrm{and}& \beta=\beta ^2,\\
\alpha= \alpha^2 + \beta &\textrm{and}& \beta=\beta ^2,\\
\alpha= \alpha^2 &\textrm{and}& \beta=\beta ^2+ \alpha,\\
\alpha= \alpha^2 + \beta&\textrm{and}& \beta=\beta ^2+ \alpha.
\end{eqnarray*}
This together with the fact that $\varepsilon(u)=\alpha+ \beta +m$ shows that $\alpha=\beta=0$ and $m=1$. Hence, $u=e_w$ for some $w \in FQ_2$, and the proof is complete.
\end{proof}

Since the link quandle of a trivial link with $n$ components is the free quandle of rank $n$, we have

\begin{cor}\label{idempotents in trivial link quandles}
If $L$ is a trivial link, then $\mathbb{Z}[Q(L)]$ has only trivial idempotents.
\end{cor}

We denote by $\Aut_{\textrm{algebra}}(\mathbf{k}[X])$ the group of $\mathbf{k}$-algebra automorphisms of $\mathbf{k}[X]$, that is, ring automorphisms of $\mathbf{k}[X]$ that are $\mathbf{k}$-linear.  Let $WB_n$ be the welded braid group on $n$-strands. See \cite{MR3689901} for a nice survey of these groups.

\begin{cor}\label{auto quandle ring free products}
$\Aut_{\textrm{algebra}}\big(\mathbb{Z}[FQ_n]\big) \cong \Aut_{\textrm{quandle}}(FQ_n)\cong WB_n$ for each $n \ge 1$.
\end{cor}

\begin{proof}
Obviously, each automorphism of $FQ_n$ induces an automorphism of $\mathbb{Z}[FQ_n]$. Conversely, if $\phi \in \Aut_{\textrm{algebra}}\big(\mathbb{Z}[FQ_n]\big)$, then $\phi$ is a bijection of the set $\mathcal{I}\big(\mathbb{Z}[FQ_n]\big)$ of all idempotents. Since $\mathbb{Z}[FQ_n]$ has only trivial idempotents, $FQ_n \cong \mathcal{I}\big(\mathbb{Z}[FQ_n]\big)$ via the map $x \mapsto e_x$, and hence $\phi$ can be viewed as an automorphism of $FQ_n$, proving the first isomorphism. The second isomorphism is a well-known result from \cite{MR1410467}.
\end{proof}
\medskip

\section{Idempotents in quandle rings of unions}\label{idempotents in unions}

Let $\{(X_i, *_i) \}_i$  be a family of quandles. Then the binary operation 
$$
x* y=\begin{cases}
x*_i y & ~\textrm{if}~ x, y \in X_i, \\
x &~\textrm{if}~  x \in X_i ~\textrm{and}~y \in X_j ~\textrm{for}~  i \neq j,
\end{cases} 
$$
turns the disjoint union  $\sqcup_i X_i$ into a quandle called the {\it union quandle}.

\begin{pro}\label{idempotent union}
Let $X= X_1 \sqcup X_2 \sqcup \cdots \sqcup  X_n$ be the disjoint union of $n \ge 2$ quandles. Then $\mathbf{k}[X]$ contains idempotents of the following form:
\begin{enumerate}
\item $\sum_{j=1}^n \alpha_j u_j$, where $u_j \in \mathbf{k}[X_j]$  is an idempotent with $\varepsilon(u_j)=1$ for each $j$ and $\sum_{i=1}^n \alpha_i=1$.
\item $\sum_{j=1}^n u_j$, where $u_i \in \mathbf{k}[X_i]$  is an idempotent with $\varepsilon(u_i)=1$ and $u_j \in \mathbf{k}[X_j]$ satisfy $u_j^2=0$ for each $j \neq i$.
\item $\sum_{j=1}^n \alpha_j (\sum_{x \in X_j} e_x)$, where $|X_j| < \infty$ and $\sum_{i=1}^n \alpha_i |X_i|=1$.
\end{enumerate}
\end{pro}

\begin{proof}
We begin by noting that if $u \in \mathbf{k}[X_i]$  and $v \in \mathbf{k}[X_k]$ for $i \ne k$, then $u v=  \varepsilon(v) u$. For assertion (1), take $w=  \sum_{j=1}^n \alpha_j u_j$, where $u_j$ is an idempotent of $\mathbf{k}[X_j]$ and  $\sum_{i=1}^n \alpha_i=1$. Then we have
$$w^2 = \sum_{i, j=1}^n \alpha_i \alpha_j  u_i u_j = \sum_{i, j=1}^n \alpha_i \alpha_j  \varepsilon(u_j) u_i =\sum_{i, j=1}^n \alpha_i \alpha_j  u_i=  \sum_{j=1}^n \alpha_j \big(\sum_{i=1}^n \alpha_i  u_i\big)= \sum_{j=1}^n \alpha_j w=w.$$
\par
For assertion (2), take $w=\sum_{j=1}^n u_j$, where $u_i$ is an idempotent in $\mathbf{k}[X_i]$ and $u_j \in \mathbf{k}[X_j]$ satisfy $u_j^2=0$ for each $j \neq i$. Since $\varepsilon(u_j)=0$ for all $j \neq i$ and $\varepsilon(u_i)=1$, it follows that
$$w^2 = \sum_{k \neq i,~~k, j=1}^n u_j u_k  + \sum_{j=1}^n u_j u_i= \sum_{k \neq i,~~k, j=1}^n \varepsilon(u_k) u_j  + \sum_{j=1}^n \varepsilon(u_i) u_j = w.$$
\par
For assertion (3), suppose that $|X_j| < \infty$ for each $j$ and take $w=\sum_{j=1}^n \alpha_j v_j$, where $v_j=\sum_{x \in X_j} e_x$ and $\sum_{i=1}^n \alpha_i |X_i|=1$. Then we see that
$$w^2= \sum_{i, j=1}^n \alpha_i \alpha_j v_iv_j=\sum_{i, j=1}^n \alpha_i \alpha_j |X_j| v_i = 
\sum _{i=1}^n \big(\sum_{j=1}^n \alpha_j |X_j| \big) (\alpha_i v_i)=\sum _{i=1}^n \alpha_i v_i=w.
$$
\end{proof}

\begin{rmk}
Note that Proposition \ref{idempotent union} holds for arbitrary families of quandles. It is interesting to see whether the proposition gives all the idempotents of the quandle ring of a union of quandles.
\end{rmk}

The union construction for two quandles has a twisted version when the quandles act on each other by automorphisms (see \cite[Proposition 11]{MR3948284}). We consider a simple case of this construction when both the quandles are trivial. Note that the automorphism group of a trivial quandle is the permutation group of the underlying set. Let $X,Y$ be trivial quandles, $f \in \Aut(X)$ and $g \in \Aut(Y)$. For $x \in X$ and $y \in Y$, setting $x*y=f(x)$ and $y*x=g(y)$ defines a quandle structure on the disjoint union $X \sqcup Y$, and we denote this quandle by $X \sqcup_{f, g} Y$. We prove a twisted version of Proposition \ref{idempotent union}.

\begin{pro}\label{idempotens in trivial unions}
Let $X$ and $Y$ be trivial quandles of orders $n$ and $m$, respectively. Let $\mathbf{k}$ be an integral domain such that characteristic of $\mathbf{k}$ is coprime to both $n$ and $m$. Let $f \in Aut(X)$ and $g \in Aut(Y)$ be automorphisms acting transitively on $X$ and $Y$, respectively. Then 
\begin{small}
$$\mathcal{I}\big(\mathbf{k}[X \sqcup_{f, g} Y]\big)= \mathcal{I}\big(\mathbf{k}[X]\big) \sqcup~ \mathcal{I}\big(\mathbf{k}[Y]\big) \sqcup \Big\{\alpha \big(\sum_{x \in X} e_x \big) + \beta \big(\sum_{y \in Y} e_y \big) ~\bigl\vert~ \alpha, \beta \in \mathbf{k}~~\textrm{~such that~}~~\alpha n+ \beta m=1 \Big\}.$$
\end{small}
\end{pro}

\begin{proof}
Note that any $u \in \mathbf{k}[X \sqcup_{f, g} Y]$ can be written uniquely as $u=v+w$, where $v= \sum_{x \in X} \alpha_x e_x \in \mathbf{k}[X]$, $w= \sum_{y \in Y} \beta_y e_y \in \mathbf{k}[Y]$ and $\alpha_x, \beta_y \in \mathbf{k}$. If $u=u^2$, then
$$v+w=v^2+w^2+vw+wv= \varepsilon(v)v + \varepsilon(w)w + \varepsilon(w) \sum_{x \in X} \alpha_x e_{f(x)}+ \varepsilon(v)\sum_{y \in Y} \beta_y e_{g(y)},$$
and consequently
$$v= \varepsilon(v)v + \varepsilon(w) \sum_{x \in X} \alpha_x e_{f(x)} \quad \textrm{and} \quad w= \varepsilon(w)w+ \varepsilon(v)\sum_{y \in Y} \beta_y e_{g(y)}.$$
Comparing coefficients give
\begin{equation}\label{eq union 1}
\alpha_x= \varepsilon(v) \alpha_x+ \varepsilon(w) \alpha_{f^{-1}(x)}
\end{equation}
and
\begin{equation}\label{eq union 2}
\beta_y = \varepsilon(w) \beta_y + \varepsilon(v) \beta_{g^{-1}(y)}
\end{equation}
for all $x\in X$ and $y \in Y$. Adding \eqref{eq union 1} for all $x \in X$ gives $\varepsilon(v)= \varepsilon(v) \varepsilon(u)$. Similarly, adding 
\eqref{eq union 2} for all $y \in Y$ gives $\varepsilon(w)= \varepsilon(w) \varepsilon(u)$. If $\varepsilon(u)=0$, then $\varepsilon(v)=\varepsilon(w)=0$, and hence $u=0$, a contradiction. So, we can assume that $\varepsilon(u)=1$, and hence at least one of $\varepsilon(v)$ or $\varepsilon(w)$ is non-zero. If $\varepsilon(v) \neq 0$, then \eqref{eq union 2} gives $\beta_y = \beta_{g^{-1}(y)}$ for all $y \in Y$. Since $g$ acts transitively on $Y$, it follows that $\beta_y= \beta$ (say) for all $y \in Y$. If $\beta=0$, then $w=0$. In this case, $u=\sum_{x \in X} \alpha_x e_x$, where $\sum_{x \in X} \alpha_x =1$, and hence $u \in  \mathcal{I}\big(\mathbf{k}[X]\big)$. If $\beta \neq 0$, then $\varepsilon(w) = m\beta \neq 0$, and \eqref{eq union 1} gives $\alpha_x= \alpha_{f^{-1}(x)}$ for all $x \in X$. Since $f$ also acts transitively on $X$, it follows that $\alpha_x= \alpha$ (say) for all $x \in X$. Thus, we have $$u= \alpha \big(\sum_{x \in X} e_x \big) + \beta \big(\sum_{y \in Y} e_y \big),$$ where $n \alpha +m \beta=1$. Similarly, if $\varepsilon(w) \neq 0$ and $\alpha = 0$, then we get $v=0$. In this case, $u=\sum_{y \in Y} \beta_y e_y$, where $\sum_{y \in Y} \beta_y =1$, and hence $u \in \mathcal{I}\big(\mathbf{k}[Y]\big)$. This completes the proof. 
\end{proof}
\medskip

\section{Concluding remarks}
We conclude with some remarks and open problems motivated by the results in the preceding sections.
\begin{enumerate}
\item All the non-zero idempotents of integral quandle rings computed in the preceding sections have augmentation value one. We believe that non-zero idempotents of integral quandle rings always have augmentation value one. This, however, fails for rings associated with idempotent quasigroups. In fact, Remark \ref{quasi group example} shows that  the ring associated with an idempotent quasigroup can have non-trivial idempotents with augmentation value zero. Furthermore, idempotents of quandle rings over $\mathbb{Z}_2$ can have augmentation value zero, for instance, this happens for all commutative quandles.
\item Let $Q(L)$ be the link quandle of a link $L$ in $\mathbb{R}^3$ and $X$ any quandle. It is well-known that the set $\Hom\big(Q(L), X\big)$ of all quandle homomorphisms extends the classical Fox colouring invariant of links. A link invariant which determines the quandle coloring invariant is called an {\it enhancement} of the quandle coloring invariant. Further, an enhancement is {\it proper} if there are examples in which the enhancement distinguishes links which have the same quandle coloring invariant. For instance, the quandle cocycle invariant is a proper enhancement arising from quandle cohomology. Since each quandle homomorphism $f:Q(L) \to X$ induces a homomorphism $\hat{f}:\mathbf{k}[Q(L)] \to \mathbf{k}[X]$ of quandle rings, it turns out that $\Hom_{\textrm{ring}}\big(\mathbf{k}[Q(L)], \mathbf{k}[X]\big)$ is an enhancement of  $\Hom\big(Q(L), X\big)$. See \cite{ENS2022} for related recent results. It is worth exploring whether this enhancement has a cohomolgical interpretation.
\item If a quandle has a subquandle of order two, then Proposition \ref{trivial subquandle idempotent} shows that  its quandle ring has non-trivial idempotents. A look at the table of quandles of order upto 35 seems to suggest that every faithful and non-latin quandle has a subquandle of order two.
\item Proposition \ref{Hopf link idempotent} shows that the quandle ring of the link quandle of the Hopf link admit non-trivial idempotents. Similarly, Corollary \ref{idempotents long knots} proves that the quandle ring of the knot quandle of the long knot has non-trivial idempotents. It is interesting to determine idempotents of quandle rings associated to other knots and links.
\item Quandle rings that have only trivial idempotents, quandle rings discussed in \cite{BPS1} and quandle rings covered by Theorem \ref{quandle covering theorem} have the property that the right multiplication by each idempotent is an automorphism. Proposition \ref{medial endomorphisms} proves that the right multiplication by an idempotent is always a ring endomorphism for medial quandles. Remark \ref{failure injectivity right mult} shows that the right multiplication by an idempotent need not be injective over the field of rationals. Further, Remark \ref{failure of distributivity} shows that idempotents fail to satisfy right-distributivity in general. In view of these observations, it would be interesting to classify quandles for which the set of all idempotents of their quandle rings over appropriate coefficients form a quandle with respect to the ring multiplication.
\item Our proof of Theorem \ref{idempotents in free products} crucially uses the fact that $FQ_2$ is the free product of one element quandles. We believe that the result holds for arbitrary free products of quandles whose quandle rings have only trivial idempotents.
\end{enumerate}
\medskip

\textbf{Acknowledgements.} 
M.E. was partially supported by the Simons Foundation Collaboration Grant 712462. The work was carried out when M.S. was visiting the University of South Florida. He thanks the USIEF for the Fulbright-Nehru Academic and Professional Excellence Fellowship that funded the visit and the University of South Florida for the warm hospitality during his stay. M.S. also acknowledges support from the SwarnaJayanti Fellowship grant.


\end{document}